\documentclass[a4paper,12pt,bibtotoc,twoside]{amsart}

\usepackage{amsfonts}
\usepackage{graphicx}
\usepackage{float}

\usepackage{url}
\usepackage{hyperref}

\usepackage{a4wide}
\usepackage{enumerate}

\usepackage[applemac]{inputenc} 

\usepackage[T1]{fontenc}
\usepackage{amsmath}
\usepackage{amssymb}
\usepackage{amsthm}

\usepackage[arrow, matrix, curve]{xy}
\usepackage{extarrows}

\usepackage{mathrsfs}
\usepackage{comment}

\usepackage{todonotes}

\newtheorem{definition}{Definition}[section]
\newtheorem{theorem}[definition]{Theorem}
\newtheorem*{mthm}{Theorem}
\theoremstyle{plain}
\newtheorem*{acknowledgement}{Acknowledgement}

\newtheorem{corollary}[definition]{Corollary}

\newtheorem{lemma}[definition]{Lemma}
\newtheorem{notation}[definition]{Notation}

\newtheorem{proposition}[definition]{Proposition}

\newtheorem{defrmk}[definition]{Definition and Remark}
\theoremstyle{definition}
\newtheorem{remark}[definition]{Remark}

\newcommand{\cE}{\mathscr{E}}
\newcommand{\OO}{\mathscr{O}}
\newcommand{\PP}{\mathbb{P}}
\newcommand{\PPr}{{\mathbb{P}^r}}
\newcommand{\PPE}{{\PP(\cE)}}
\newcommand{\cC}{\mathscr{C}}

\newcommand{\cL}{\mathscr{L}}
\newcommand{\cH}{\mathscr{H}}
\newcommand{\cF}{\mathscr{F}}

\newcommand{\cHom}{\mathscr{H}\!om}
\newcommand{\Tor}{\text{\upshape{Tor}}}

\newcommand{\Aut}{\text{\upshape{Aut}}}
\newcommand{\codim}{\text{\upshape{codim}}}
\newcommand{\Hom}{\text{\upshape{Hom}}}
\newcommand{\Pf}{\text{\upshape{Pf}}}
\newcommand{\im}{\text{\upshape{im}}}

\newcommand{\Syz}{\text{\upshape{Syz}}}
\newcommand{\rank}{\text{\upshape{rank}}}

\newcommand{\Pic}{\text{\upshape{Pic}}}

\begin{document}

\numberwithin{equation}{section}

\title{Syzygies of 5-gonal Canonical Curves}

\author{Christian Bopp}
\address{Universit\"at des Saarlandes}
\email{bopp@math.uni-sb.de}

\begin{abstract}
We show that for $5$-gonal curves of odd genus $g\geq 13$ and even genus $g\geq 28$ 
 the $\lceil \frac{g-1}{2}\rceil$-th syzygy module of the curve is not determined by the syzygies of the scroll swept out by the special pencil of degree $5$.
\end{abstract}

\maketitle


\section{Introduction}
In this article we study minimal free resolutions of the coordinate ring $S_C$ of $5$-gonal canonical curves $C\subset \PP^{g-1}$. The \emph{gonality} of a curve $C$ is defined as the minimal degree of a nonconstant map
$
C{\longrightarrow} \PP^1.
$
\newline
A pencil of degree $k$ on a canonically embedded curve $C$ of genus $g$, defining a degree $k$ map $C\to \PP^1$ sweps out a rational normal scroll  $X$ of dimension $d=k-1$ and degree $f=g-k+1$.
 It follows that the linear strand of $X$ is a subcomplex of the linear strand of $C$. To be more precise the scroll contributes with an Eagon-Northcott complex of length $g-k$ to the linear strand of the curve. 
 \newline
 It is well known that the gonality of a general genus $g$ curve is precisely $\lceil \frac{g+2}{2}\rceil$. 
 Thus if $X$ is the scroll swept out by a pencil of minimal degree on a general canonical curve, then
  $\beta_{n,n+1}(X)=0$, where $n=\lceil \frac{g-1}{2}\rceil$.
 For odd genus $g$ and ground field of characteristic $0$, Voisin and Hirschowitz-Ramanan (see \cite{V2} and \cite{Hirs-Ram}) showed that the locus 
 $$
 \mathcal{K}_g:=\{C\in \mathcal{M}_g\  |\  \beta_{n,n+1}(C)\neq 0\}
 $$
  defines an effective divisor in the moduli space of curves, the so-called \emph{Koszul divisor}. \newline
 On the Hurwitz-scheme $\cH_{g,k}$ a natural analogue of the Koszul divisor could be the following
 $$
 \mathcal{K}_{g,k}:=\{C\in \mathcal{H}_{g,k}\ | \ \beta_{n,n+1}(C)>\beta_{n,n+1}(X)\}.
 $$
 For odd genus $ \mathcal{K}_{g,k}$ is a divisor in the Hurwitz-scheme if 
 $
 \beta_{n,n+1}(C)=\beta_{n,n+1}(X)
 $
 holds for a general curve $C\in \cH_{g,k}$ (see \cite{Hirs-Ram}), 
 which we verified computationally for genus $g<13$ and $k<\lceil\frac{g-2}{2}\rceil$.
\newline
 We will show that $\mathcal{K}_{g,5}$ is not a divisor for odd $g\geq 13$ by proving the following theorem.
 \begin{mthm}
Let $C$ be a $5$-gonal canonical curve of genus $g$  and $n= \lceil \frac{g-1}{2}\rceil$. Then
$$\beta_{n,n+1}(C)>\beta_{n,n+1}(X)$$ for odd genus $g\geq 13$ and even genus $g\geq 28$.
\end{mthm}

The proof is based on the techniques  introduced in \cite{Sch}, which we recall in Section \ref{section2}. First we resolve the curve $C$ as an $\OO_\PPE$-module, where $\PPE$ is the bundle associated to the rational normal scroll swept out by the $g^1_5$. 
In the next step we resolve the $\OO_\PPE$-modules occurring in this resolution by Eagon-Northcott type complexes. An iterated mapping cone construction then gives a non-minimal resolution of $C\subset \PP^{g-1}$. 
By determining the ranks of the maps which give rise to non-minimal parts in the iterated mapping cone we can decide whether the curve has extra syzygies. 
In the last section we discuss the genus $13$ case in detail.

\section{Scrolls, Pencils and Canonical Curves}\label{section2}
In this section we briefly summarize the connections between pencils on canonical curves and rational normal scrolls. Most of this section follows Schreyer's article \cite{Sch}.
\begin{definition}
Let $e_1 \geq e_2 \geq \dots \geq e_d\geq 0$ be integers, 
$\cE=\OO_{\PP^1}(e_1)\oplus\dots\oplus\OO_{\PP^1}(e_d)$ and let 
$\pi: \PPE\to \PP^1$ be the corresponding $\PP^{d-1}$-bundle.\newline
A \emph{rational normal scroll} $X=S(e_1,\dots,e_d)$ of  type $(e_1,\dots,e_d)$ is the image of 
$$
j:\PP(\cE)\to \PP H^0(\PPE,\OO_\PPE(1))= \PP^r
$$ \vspace{-1mm}
where $r=f+d-1$ with $f=e_1+\dots+e_d\geq2$.
\end{definition}
In \cite{H} it is shown that
the variety $X$ defined above is a non-degenerate $d$-dimensional variety of minimal degree
$\deg X=f=r-d+1=\codim X+1$.
If  $e_1,\dots,e_d>0$, then $j:\PP(\cE)\to X\subset \PP H^0(\PPE,\OO_\PPE(1))= \PP^r$ is an isomorphism. Otherwise it is a resolution of singularities and since the singularities of $X$ are rational, we can consider $\PPE$ instead of $X$ for most cohomological considerations. \newline
It is furthermore shown, that the Picard group $\Pic(\PPE)$ is generated by the ruling $R=[\pi^*\OO_{\PP^1}(1)]$ and the hyperplane class
$H=[j^*\OO_{\PP^r}(1)]$ with intersection products
$$
H^d=f, \ \ H^{d-1}\cdot R=1, \ \ R^2=0.
$$

\begin{remark}[\cite{Sch}]\label{Xample}
For $a\geq0$ we have
$ H^0(\PPE,\OO_\PPE(aH+bR))\cong H^0(\PP^1,S_a(\cE)(b))$.
Thus we can compute the cohomology of the line bundle $\OO_\PPE(aH+bR)$ explicitly. \newline
 If we denote by 
$\Bbbk[s,t]$ the coordinate ring of $\PP^1$ and by $\varphi_i\in H^0(\PPE,\OO_\PPE(H-e_iR))$ the basic sections, then we can identify sections $\psi\in H^0(\OO_\PPE,\OO_\PPE(aH+bR))$ with homogeneous polynomials \vspace{-2mm}
$$
\psi = \sum_\alpha P_\alpha(s,t)\varphi_1^{\alpha_1}\dots\varphi_d^{\alpha_d}
$$
of degree $a=\alpha_1+\cdots+\alpha_d$ in $\varphi_i$'s and with polynomial coefficients $P_\alpha \in \Bbbk[s,t]$ of degree $\deg P_\alpha= \alpha_1e_1+\cdots+\alpha_de_d+b$. Thus for 
$a\cdot \min\{e_i\}+b\geq -1$ we get
$$
h^0(\PPE, \OO_\PPE(aH+bR))=f\binom{a+d-1}{d}+(b+1)\binom{a+d-1}{d-1}.
$$
In particular $\OO_\PPE(aH+bR)$ is globally generated if $a \geq0$ and $a\cdot\min\{e_i\}\geq 0$.
\end{remark}
Next we want to describe how to resolve line bundles $\OO_\PPE(aH+bR)$ by $\OO_{\PP^r}$-modules.
If we denote by $\Phi$ the $2\times f$ matrix with entries in $H^0(\PPE,\OO_\PPE(H))$  obtained from the multiplication map
$$
H^0(\PPE,\OO_\PPE(R))\otimes H^0(\PPE,\OO_\PPE(H-R))\longrightarrow H^0(\PPE,\OO_\PPE(H)),
$$
then the equations of $X$ are given by the $2\times 2$ minors of $\Phi$. We define 
$$
F:= H^0(\PPE,\OO_\PPE(H-R))\otimes \OO_{\PP^r}= \OO_{\PP^r}^f\ \ \text{and} \ \ 
G:=H^0(\PPE,\OO_\PPE(R))\otimes \OO_{\PP^r}= \OO_{\PP^r}^2
$$
and regard $\Phi$ as a map $\Phi: F(-1)\to G$. For $b\geq -1$, we
consider the \emph{Eagon-Northcott type} complex $\cC$, whose $j^{\text{th}}$ term is defined by
$$
\cC^b_j=\begin{cases}\bigwedge^j F\otimes S_{b-j}G\otimes \OO_\PPr(-j),\ &{\text{for} \ 0 \leq j\leq b}\\ 
\bigwedge^{j+1}F\otimes D_{j-b-1}G^*\otimes  \OO_\PPr(-j-1), &{\text{for}\ j\geq b+1\ \ }\end{cases}
$$
and whose differentials $\delta_j:\cC^b_j\to\cC^b_{j-1}$ are given by the multiplication with $\Phi$ for $j\neq b+1$ and by $\bigwedge^2\Phi$ for $j=b+1.$

\begin{theorem}\label{Sch1.6}
 The Eagon-Northcott type complex 
$\cC^b(a):=\cC^b\otimes \OO_\PPr(a),$ defined above, gives a minimal free resolution of 
$\OO_\PPE(aH+bR)$ as an $\OO_\PPr$-module.
\end{theorem}
\begin{proof}
See \cite[section 1]{Sch}.
\end{proof}

Now let $C \subset \PP^{g-1}$ be a canonically embedded curve of genus $g$ and let further 
$$g^1_k=\{D_\lambda\}_{\lambda \in \PP^1}\subset |D|$$ 
be a pencil of divisors of degree $k$. If we denote by $\overline{D_\lambda}\subset \PP^{g-1}$ the linear span of the divisor, then
$$
X=\bigcup_{\lambda\in \PP^1}\overline{D_\lambda}\subset \PP^{g-1}
$$
 is a $(k-1)$-dimensional rational normal scroll of degree $f=g-k+1$ (see e.g. \cite{EH}). Conversely if $X$ is a rational normal scroll of degree $f$ containing a canonical curve, then the ruling on $X$ cuts out a pencil of divisors $\{D_\lambda\}\subset |D|$ such that $h^0(C,\omega_C\otimes \OO_C(D)^{-1})=f.$ 
 
 \begin{notation}
During the rest of this article $C\subset \PP^{g-1}$ will denote a canonical curve with a basepoint free $g^1_k.$ The variety $X=S(e_1,\dots,e_d)$ is the scroll of degree $f$ and dimension $d=k-1$ defined by this pencil and 
$\PPE$ will denote the $\PP^{d-1}$-bundle corresponding to $X.$ 
\end{notation}

The next important theorem due to Schreyer explains how to obtain a free resolutions of $k$-gonal canonical curves $C\subset \PP^{g-1}$ by an iterated mapping cone construction.

\begin{theorem}\label{Sch4.4}
\begin{enumerate}[i\upshape{)}]
\item
$C\subset \PPE$ has a resolution $F_\bullet$ of type 
$$0\longrightarrow \OO_\PPE(-kH+(f-2)R) \longrightarrow \sum_{j=1}^{\beta_{k-3}}\OO_\PPE(-(k-2)H+a^{(k-3)}_jR)\longrightarrow$$ \vspace{-5mm}
$$\dots\longrightarrow \sum_{j=1}^{\beta_1}\OO_\PPE(-2H+a_j^{(1)}R)\longrightarrow \OO_\PPE \longrightarrow \OO_C\longrightarrow 0\ \ \ \ \ \ \ \ \ \ \ \ $$
with $\beta_i=\frac{i(k-2-i)}{k-1}\binom{k}{i+1}.$ 
\item The complex $F_\bullet$ is self dual, i.e.,
\begin{center}$\cHom(F_\bullet, \OO_\PPE(-kH+(f-2)R))\cong F_\bullet$\end{center}
\item If all $a_k^{(j)}\geq -1,$ then an iterated mapping cone 
$$\left[  \left[ \dots \left [ \cC^{(f-2)}(-k)\longrightarrow \sum_{j=1}^{\beta_{k-3}}\cC^{(a_j^{(k-3)})}(-k+2)\right]\longrightarrow \dots \right]\longrightarrow \cC^0\right]$$
 gives a, not necessarily minimal, resolution of $C$ as an $\OO_{\PP^{g-1}}$-module. 
\end{enumerate}
\end{theorem}

\begin{proof}
See \cite[Corollary (4.4)]{Sch} and \cite[Lemma (4.2)]{Sch}.
\end{proof}

\begin{remark}\label{5gonrels}
The $a_i^{(k)}$'s in part $i)$ of the theorem above satisfy certain linear equations obtained from the Euler characteristic of the complex $F_\bullet$. 
In particular Schreyer shows (see \cite[Section 6]{Sch}), that for $5$-gonal curves $C$ 
\begin{center}
 $\sum_{i=1}^5 a_i=2g-12 \ \ \text{and}\ \ 
a_i+b_i=f-2,$ 
\end{center}
where $a_i:=a_i^{(1)}$, $b_i:=a_i^{(2)}$ 
and $f=g-4$ is the degree of the rational normal scroll swept out by the $g^1_5$ on $C$. 
\end{remark}

\begin{definition}
We call a partition $(e_1,\dots,e_d)$ \emph{balanced} if $\max_{i,j}|e_i-e_j|\leq 1$.
A rational normal scroll $X=S(e_1,\dots,e_d)$ of dimension $d$ is said to be of \emph{balanced type} if 
$(e_1,\dots,e_d)$ is a balanced partition. \\
With the same notation as in the remark above we say that a $5$-gonal canonical curve 
\emph{satisfies the balancing conditions}
if 
\begin{itemize}
\item the scroll $X$ of dimension $4$ swept out by the $g^1_5$ on $C$ is of balanced type and
\item the partition $(a_1,\dots ,a_5)$ of $2g-12$ is balanced.
\end{itemize}
\end{definition}

\begin{remark}\label{scrollBalanced}
One can show (see e.g. \cite[3.3]{viviani}) that a generic $k$-gonal curve sweps out a scroll of balanced type. The proof uses Ballico's Theorem (see \cite{Bal} ) and the fact that the type of the scroll is uniquely determined by the  $g^1_k$ (see \cite[(2.4)]{Sch}).
\end{remark}

For a section $\Psi: \OO_X(-H+bR)\longrightarrow \OO_X(aR)$ in $H^0(X,\OO_X(H-(b-a)R))$ the induced comparison maps $\rho_\bullet : \cC_\bullet^b(-1)\longrightarrow\cC_\bullet^a$ are determined by $\Psi$ up to homotopy. 
For example
\begin{center}$
\Hom(\cC^b_{a+1}(-1),\cC^a_{a+2})=\Hom(\cC^b_a(-1),\cC^a_{a+1})=0
$\end{center}
by degree reasons, and therefore the $(a+1)^{st}-$comparison map $\rho_{a+1}$ is uniquely determined by $\Psi$ (not only up to homotopy).
The following lemma is due to Martens and Schreyer.

\begin{lemma}\label{remMartens}\label{martens}
If $\Hom(\cC^b_{j}(-1),\cC^a_{j+1})=\Hom(\cC^b_{j-1}(-1),\cC^a_{j})=0$, then  
the $j^{\text{th}}$-comparison map 
$\rho_j:\cC^b_j\to\cC^a_j$ is given (up to a scalar factor) by the composition 
\begin{small}
$$
\rho_j: \bigwedge^jF\otimes S_{b-j}G\longrightarrow \bigwedge^jF\otimes S_{b-j}G\otimes S_{j-a-1}G\otimes D_{j-a-1}G^*
$$ \vspace{-4mm}$$
 \xlongrightarrow{id\otimes mult\otimes id}\bigwedge^j F\otimes S_{b-a-1}G\otimes D_{j-a-1}G^* 
\xrightarrow{id\otimes\Psi\otimes id}\bigwedge^{j}F\otimes F\otimes D_{j-a-1}G^*
$$ \vspace{-4mm} $$
 \xrightarrow{\wedge\otimes id}\bigwedge^{j+1}F\otimes D_{j-a-1}G^*.
$$
\end{small}
\end{lemma}
\begin{proof}
In \cite{MS}[appendix] this is shown for the $(a+1)^{st}$ comparism map but the proof immediately generalizes to our situation.
\end{proof}

\section{Proof of the Main Theorem}\label{pmthm}
The aim of this section is to prove Theorem \ref{mthm} and then to show that a general $5$-gonal curve satisfies the balancing conditions.

\begin{theorem}\label{mthm}
Let $C\subset \PP^{g-1}$ be a general $5$-gonal curve of genus $g$ satisfying the balancing conditions. Let further $X$ be the scroll swept out by the $g^1_5$ and $n=\lceil\frac{g-1}{2}\rceil$, then 
\begin{center}
$\beta_{n,n+1}(C)>\beta_{n,n+1}(X)\  \text{for odd}\  g\geq 13\  \text{and even}\  g\geq 28$.
\end{center} 
\end{theorem}
Throughout this section, $C\subset \PP^{g-1}$ will be a $5$-gonal canonical curve of genus $g$.
In this case $X$ is a $d=4$ dimensional rational normal scroll of degree $f=g-4$. 
Recall from \ref{Sch4.4} and \ref{5gonrels} that $C\subset \PPE$ has a resolution of the form
$$
\OO_\PPE(-5H+(f-2)R)\to\sum_{i=1}^5\OO_\PPE(-3H+b_iR)
\xrightarrow{\Psi}\sum_{i=1}^5\OO_\PPE(-2H+a_iR)\to\OO_\PPE\to\OO_C
$$
where $\sum_{i=1}^5a_i=2g-12,$  $a_i+b_i=f-2$. \\
The matrix $\Psi$ is skew symmetric by the structure theorem for Gorenstein ideal in codimension $3$ and the $5$ Pfaffians of $\Psi$ generate the ideal of $C$ (see \cite{BE3}). \newline
As in Theorem \ref{Sch1.6}, we denote by 
$F=H^0(\PPE,\OO_\PPE(H-R))\otimes \OO_{\PP^{g-1}}\cong \OO_{\PP^{g-1}}^f$ and by $G=H^0(\PPE,\OO_\PPE(R))\otimes \OO_{\PP^{g-1}}\cong \OO_{\PP^{g-1}}^2$. By abuse of notation, we will also refer to the vector spaces 
$H^0(\PPE,\OO_\PPE(H-R))$ and $H^0(\PPE,\OO_\PPE(R))$ by $F$ and $G$, respectively. \\
Resolving the $\OO_\PPE$-modules occurring in the minimal resolution of $C\subset \PPE$, we get 
\begin{center}
\hfil
\begin{small}
\begin{xy}
\xymatrix@R=4mm{
\sum_{i=1}^5\OO_\PPE(-3H+b_iR)\ar[r]^{\Psi\ \ \ } &\sum_{i=1}^5\OO_\PPE(-2H+b_iR)\ \ \ \ \ \ \ \ \  \\
\vdots\ar[u] &\vdots \ar[u] \\
\sum_{i=1}^5\cC^{b_i}_{n-2}(-3)\ar[u]\ar[r]^{\psi_{n-2}} &
\sum_{i=1}^5\cC^{a_i}_{n-2}(-2) \ar[u] \\
\sum_{i=1}^5\cC^{b_i}_{n-1}(-3)\ar[u]\ar[r]^{\psi_{n-1}}&
\sum_{i=1}^5\cC^{a_i}_{n-1}(-2) \ar[u]\\
\vdots \ar[u]& \vdots\ar[u]
}
\end{xy}
\end{small}
\hfil
\end{center}

The Betti number $\beta_{n,n+1}(X)$ is given by $\rank(\cC^0_n)=n\cdot\binom{f}{n+1}$ and
the rank of the $(n-2)^{\text{th}}$ comparison map 
$$
\psi:=\psi_{n-2}:\sum_{i=1}^5\cC^{b_i}_{n-2}(-3)\longrightarrow \sum_{i=1}^5\cC^{a_i}_{n-2}(-2)
$$
in the diagram above determines the Betti number $\beta_{n,n+1}(C)$. \newline
We always have $\rank(\sum_{i=1}^5\cC^{b_i}_{n-2})\leq\rank(\sum_{i=1}^5\cC^{a_i}_{n-2})$ (where equality holds for odd genus). \newpage
 Furthermore 
 \begin{center}$ \text{min}\{b_i\}\geq n-2\geq \text{max}\{a_i\}$\end{center}
holds for curves of genus odd $g=2n+1\geq 13$ (or even genus $g=2n\geq 28$)  satisfying the balancing conditions and therefore
$$
\cC^{b_i}_{n-2}(-3)=\bigwedge^{n-2}F\otimes S_{b_i-n+2}G(-n-1)\ \  \text{and} \ \ 
\cC^{a_i}_{n-2}(-2)=\bigwedge^{n-1}F\otimes D_{n-a_i-3}G^*(-n-1).
$$
This means in particular, that the map $\psi$ has entries in the field $\Bbbk$ and therefore
$$
\beta_{n,n+1}(C)=\beta_{n,n+1}(X)+\dim\ker(\psi).
$$

We restrict ourselves to curves of odd genus since the theorem is proved in exactly the same way for even genus.
For curves of odd genus $g=2n+1$, we distinguish $5$ types of curves satisfying the balancing conditions. 
These types depend on the congruence class of $n$ modulo $5$, i.e., on the block structure of the skew symmetric matrix $\Psi$. 
\begin{description}
\item[\rm\bf Type I] $(a_1,\dots,a_5)=(a,a,a,a,a),(b_1,\dots,b_5)=(b,b,b,b,b) \Leftrightarrow$
$a=4r+2$, $b=6r+3$ and $n=5r+5$ for $r\geq1.$
\item[\rm\bf Type II] $(a_1,\dots,a_5)=(a-1,a,a,a,a),(b_1,\dots,b_5)=(b+1,b,b,b,b) \Leftrightarrow$ $a=4r-1$, $b=6r-2$ and $n=5r+1$ for $r\geq1.$
\item[\rm\bf Type III] $(a_1,\dots,a_5)=(a-1,a-1,a,a,a),(b_1,\dots,b_5)=(b+1,b+1,b,b,b) \Leftrightarrow$
 $a=4r$, $b=6r-1$ and $n=5r+2$ for $r\geq1.$
\item[\rm\bf Type IV] $(a_1,\dots,a_5)=(a,a,a,a+1,a+1),(b_1,\dots,b_5)=(b,b,b,b-1,b-1) \Leftrightarrow$
$a=4r$, $b=6r+1$ and $n=5r+3$ for $r\geq1.$
\item[\rm\bf Type V] $(a_1,\dots,a_5)=(a,a,a,a,a+1),(b_1,\dots,b_5)=(b,b,b,b,b-1) \Leftrightarrow$ $a=4r+1$, $b=6r+2$ and $n=5r+4$ for $r\geq1.$
\end{description}
\begin{proof}[Proof of Theorem \ref{mthm}]
\ \\
Since the proof of the theorem is similar for all different types above, we will only carry out the proof for curves of type II (leaving the other cases to the reader).  \newline
We show that the map \vspace{-2mm}
$$
\psi:\bigg( \cC^{(b+1)}_{n-2}\oplus\sum_{i=1}^4\cC^b_{n-2}\bigg)(-3)\longrightarrow 
\bigg(\cC^{(a-1)}_{n-2}\oplus \sum_{i=1}^4\cC^a_{n-2}\bigg)(-2)
$$
induced by the skew-symmetric matrix 
$$\Psi:
\overset{\OO_\PPE(-3H+(b+1)R)}{\underset{\OO_\PPE(-3H+bR)^{\oplus 4}}{\oplus}}
\longrightarrow
\overset{\OO_\PPE(-2H+(a-1)R)}{\underset{\OO_\PPE(-2H+aR)^{\oplus 4}}{\oplus}}\ .
$$
has a non-trivial decomposable element in the kernel.
Note that the map 
$$\Psi_{(11)}:\OO_\PPE(-3H+(b+1)R)\longrightarrow\OO_\PPE(-3H+(a-1)R)$$
is zero by the skew-symmetry of $\Psi.$
Thus it is sufficient to find an element in the kernel the map
$\psi_{(41)}: \cC^{(b+1)}_{n-2}(-3)\longrightarrow
\sum_{i=1}^4\cC^{(a)}_{n-2}(-2)$,
 induced by the submatrix 
$$\Psi_{(41)}:\OO_\PPE(-3H+(b+1)R) \longrightarrow \OO_\PPE(-2H+aR)^{\oplus 4}$$
of the matrix $\Psi$.
By Lemma \ref{remMartens}, the map $\psi_{(41)}$ is uniquely determined and is given as the composition
$$
\cC_{n-2}^{b+1}(-3)=\bigwedge^{n-2}F\otimes S_{b-n+3}G=\bigwedge^{n-2}F\otimes S_{n-a-2}G \hookrightarrow 
\bigwedge^{n-2}F\otimes S_{n-a-2}G\otimes S_{n-a-3}G\otimes D_{n-a-3}G^*
$$ \vspace{-4mm}
$$
\longrightarrow  \bigwedge^{n-2}F\otimes S_{2n-2a-5}G\otimes D_{n-a-3}G^* 
\xlongrightarrow{id\otimes \Psi_{(41)}\otimes id} \bigwedge^{n-2}F\otimes F^{\oplus 4}\otimes D_{n-a-3}G^* 
$$ \vspace{-4mm}
$$
\xlongrightarrow{\wedge\otimes id} \bigg(\bigwedge^{n-1}F\bigg)^{\oplus 4}\otimes D_{n-a-3}G^*= \sum_{i=1}^4\cC_{n-2}^a(-2)\ .
$$
Since the multiplication map  $S_{n-a-2}G\otimes S_{n-a-3}G\longrightarrow S_{2n-2a-5}G$ is not injective, we show that the existence of an $f\in \bigwedge^{n-2}F$ and a $g\in S_{n-a-2}G$ such that $f\wedge \Psi_{(41)}(g\cdot g')=0$ for all $g'\in S_{n-a-3}G$.\newline
To this end, let $g\in S_{n-a-2}G$ be an arbitrary element and let $\{g'_1,\dots, g'_{n-a-2}\}$ be a basis of $S_{n-a-3}G$.
For $i=1,\dots, (n-a-2)$, we define 
$$
(f_1^{(i)},f_2^{(i)},f_3^{(i)},f_4^{(i)})^t:=\Psi_{(41)}(g\cdot g'_i)\in F^4.
$$
Note that the set $\{f_j^{(i)}\}\subset F$ is a linearly independent subset of $F$
 since $C$, and therefore $\Psi$ is general.
Thus, since 
$$
n-2=5r-1\geq 4\cdot \dim_\Bbbk(S_{n-a-3}G)=4(n-a-2)=4r
$$
holds for all $r\geq 1$, we find a nonzero element $f$ of the form
$$
f=f_1^{(1)}\wedge f_1^{(2)}\wedge \dots \wedge f_{n-a-2}^{(3)}\wedge f_{n-a-2}^{(4)}\wedge
\tilde{f} \in \bigwedge^{n-2}F
$$
for some $\tilde{f}\in \bigwedge^{r-1}F$. By construction $(f\otimes g,0,0,0,0)^t$ lies in the kernel of $\psi$.
\end{proof}

We can immediately generalize Theorem \ref{mthm}.
\begin{theorem}
Let $C\subset \PP^{g-1}$ be a general $5$-gonal canonical curve satisfying the balancing conditions. Then
$$
\beta_{n+c,n+c+1}(C)>\beta_{n+c,n+c+1}(X) \ \ \text{for odd genus}\ g=2n+1\geq 30c+13
$$
$$
\beta_{n+c,n+c+1}(C)>\beta_{n+c,n+c+1}(X) \ \ \text{for even genus}\ g=2n\geq 30c+28.
$$
\end{theorem}
\begin{proof}
The Betti number $\beta_{n+c,n+c+1}(C)$ is determined by the rank of the map
$$\psi_{n-2+c}:\sum_{i=1}^5\cC^{b_i}_{n-2+c}(-3)\to \sum_{i=1}^5\cC^{a_i}_{n-2+c}(-2)$$
and $\rank(\sum_{i=1}^5\cC^{b_i}_{n-2+c})\leq \rank( \sum_{i=1}^5\cC^{a_i}_{n-2+c})$.
By the same construction as in the proof of Theorem \ref{mthm}, we find decomposable elements in the kernel of $\psi_{n-2+c}$ if $r\geq 3c+1$ (for odd genus) or $r\geq 3c+3$ (for even genus).
\end{proof}

Next we want to show that the generic $5$-gonal curve satisfies the balancing conditions. 
The balancing conditions on a $5$-gonal canonical curve are open conditions.
It is therefore sufficient to find an example for each $g$.

\begin{proposition}\label{existBalanced}
For any odd $g\geq 13$ (and even $g\geq 28$), there exists a smooth and irreducible $5$-gonal canonical genus $g$ curve satisfying the balancing conditions. 
\end{proposition}

\begin{proof}
We illustrate  the proof for odd genus curves of type II:\newline
To this end let $X=S(e_1,\dots,e_4)\cong \PPE$ with $e_1\geq\dots\geq e_4$ be a fixed $4$-dimensional rational normal scroll of balanced type and of degree $f=g-4$. 
Let further 
$$
(a_1,\dots,a_5)=(a-1,a,a,a,a)\ \  \text{and} \ \  (b_1,\dots,b_5)=(b+1,b,b,b,b)
$$
be balanced partitions such that $\sum_{i=1}^5 a_i=2g-12$, $g=2n+1$, $n=5r+1$, $a=4r-1$ and $a_i+b_i=g-6$.
\newline
We consider a general skew-symmetric morphism 
\begin{large}
$$
\underbrace{\overset{\OO_\PPE(-3H+(b+1)R)}{\underset{\OO_\PPE(-3H+bR)^{\oplus 4}}{\oplus}}}_{=: \cF}
 \xlongrightarrow{\Psi}
\underbrace{\overset{\OO_\PPE(-2H+(a-1)R)}{\underset{\OO_\PPE(-2H+aR)^{\oplus 4}}{\oplus}}}_{=:\cF^*\otimes \cL}
$$
\end{large}
If $\bigwedge^2\cF^*\otimes \cL$ is globally generated, then it follows by a Bertini type theorem (see e.g. \cite{ottaviani}) that the scheme $\Pf(\Psi)$ cut out by the Pfaffians of $\Psi$ is smooth of codimension $3$ or empty. \newline
Recall from Remark \ref{Xample}, that a line bundle $\OO_\PPE(H+cR)$ is globally generated if 
$e_4+c\geq 0$. Thus since for $r\geq 1$
$$
\min\{e_i\}=\left\lfloor \frac{g-4}{4}\right\rfloor = \left\lfloor \frac{10r-1}{4}\right\rfloor = 2r+\left\lfloor \frac{2r-1}{4}\right\rfloor  \geq (b-a+1)=2r,
$$
we conclude that in this case
$$
\bigwedge^2\cF^*\otimes \cL=\OO_\PPE(H-(b-a+1)R)^{\oplus 4}\oplus\OO_\PPE(H-(b-a)R)^{\oplus 6}
$$
is globally generated. It follows that $C=\Pf(\Psi)$ is smooth of codimension $3$ or empty.
Since the iterated mapping cone construction gives a resolution of $C\subset \PP^{g-1}$ it follows that $C$ is a non-empty (and therefore smooth) arithmetically Cohen-Macaulay scheme. In particular we have a surjective map 
$$ \vspace{-2mm}
\underbrace{H^0(\PP^{g-1},\OO_{\PP^{g-1}})}_{\cong \Bbbk} \to H^0(C,\OO_{C}) \to 0
$$
and therefore $C$ is smooth and connected and hence an irreducible curve.
\newline
Doing the same for curves of type I, III, IV and V  and the even genus cases the result follows for all genera except for $g=15$ (in this case $\bigwedge^2\cF^*\otimes\cL$ is not globally generated). 
For the $g=15$ case one can verify the smoothness by using \emph{Macaulay2} (see \cite{M2}).
\end{proof}

\begin{remark}
The \emph{Macaulay2} code for the $g=15$ case, as well as the code that verifies the statements in the next section can be found here:\newline
\begin{footnotesize}
\url{http://www.math.uni-sb.de/ag-schreyer/joomla/images/data/computeralgebra/fiveGonalFile.m2}
\end{footnotesize}
\end{remark}
\newpage
\begin{corollary}
 $
 \mathcal{K}_{g,5}:=\{C\in \mathcal{H}_{g,5}\ | \ \beta_{n,n+1}(C)>\beta_{n,n+1}(X)\}
 $
 is a non empty and dense subset of $\cH_{g,5}$ for odd genus $g\geq 13$ and even genus $g\geq 28$.
\end{corollary}
By the semi-continuity on the Betti numbers, it follows in fact that 
$$\{C\in \mathcal{H}_{g,5}\ | \ \beta_{n,n+1}(C)=\beta_{n,n+1}(X)\}=\emptyset .$$

\section{A First Example}
In this section, we discuss the case of a general $5$-gonal canonical curve of genus $13$. The rational normal scroll $X$ swept out by the $g^1_5$ on $C$ is therefore a $4$-dimensional scroll of type $S(3,2,2,2)$ and degree $f=9$. 
The curve $C\subset \PPE$ has a resolution of the form
$$
0\to\OO_\PPE(-5H+7R)\to\
\overset{\OO_\PPE(-3H+5R)}{\underset{\OO_\PPE(-3H+4R)^{\oplus 4}}{\oplus}}
\overset{\Psi}{\longrightarrow} 
\overset{\OO_\PPE(-2H+2R)}{\underset{\OO_\PPE(-2H+3R)^{\oplus 4}}{\oplus}}
\to \OO_\PPE\to \OO_C\to 0
$$
where $\Psi$ is a skew-symmetric matrix with entries as indicated below 
\begin{equation}\label{g13Psi}
(\Psi) \thicksim
\begin{Small}
\left(
\begin{array}{ccccc}
  0   &(H-2R) &(H-2R) &(H-2R) &(H-2R)   \\
    &0   &(H-R) &(H-R) & (H-R)   \\
    &   &0 & (H-R) & (H-R)   \\
    &  &   & 0 & (H-R) \\
    &  &  &  & 0
\end{array}
\right) \end{Small}
\end{equation}
We resolve the $\OO_\PPE$-modules in the resolution above by Eagon-Northcott type complexes and determine the rank of the maps which give rise to non-minimal parts in the iterated mapping cone. As in section \ref{pmthm}, we denote by $F=H^0(\PPE,\OO_\PPE(H-R))$ and by 
$G=H^0(\PPE,\OO_\PPE(R))$.
$$
\begin{xy}
\xymatrix@R=6mm{
&\\  
 \sum_{i=1}^5\cC^{b_i}_3(-3) =\overset{\bigwedge^3F\otimes S_2G(-6)}{\underset{(\bigwedge^3F\otimes S_1G(-6))^{\oplus 4}}{\oplus}} 
\ar[u]\ar[r]^{\psi_3} &\sum_{i=1}^5\cC^{a_i}_3(-2)=\overset{\bigwedge^{ 4}F(-6)}{\underset{(\bigwedge^4F(-5))^{\oplus 4}}{\oplus}}\ar[u]\\
\sum_{i=1}^5\cC^{b_i}_4(-3)=\overset{\bigwedge^4F\otimes S_1G(-7)}{\underset{(\bigwedge^4F(-7))^{\oplus 4}}{\oplus}} \ar[u]\ar[r]^{\psi_4}
&\sum_{i=1}^5\cC^{a_i}_4(-2)=\overset{\bigwedge^5F\otimes D_1G^*(-7)}{\underset{(\bigwedge^5F(-7))^{\oplus 4}}{\oplus}}\ar[u] \\
\sum_{i=1}^5\cC^{b_i}_5(-3)=\overset{\bigwedge^5F(-8)}{\underset{(\bigwedge^5F(-9))^{\oplus 4}}{\oplus}} 
\ar[u]\ar[r]^{\psi_5} &\sum_{i=1}^5\cC^{a_i}_5(-2)= \overset{\bigwedge^6F\otimes D_2G^*(-8)}{\underset{(\bigwedge^6F\otimes D_1G^*(-8))^{\oplus 4}}{\oplus}}\ar[u]\\
\ar[u] & \ar[u]
}
\end{xy}
$$
Note that by degree reasons, the maps indicated above are the only ones which give rise to non-minimal parts in the iterated mapping cone. \newline
By the Gorenstein property of canonical curves, it follows that the maps 
$$
\psi_3': (\bigwedge^3F\otimes S_1G)^{\oplus 4}(-6)\to \bigwedge^4F(-6)\ \  \text{and} \ \ 
\psi_5': \bigwedge^5F(-8)\to (\bigwedge^6F\otimes D_1G^*)^{\oplus 4}(-8)
$$
are dual to each other and one can easily check the surjectivity of $\psi_3'$. It remains to compute the rank of $\psi_4$.

\begin{proposition}
Let $\Psi$ be a general skew symmetric matrix with entries as indicated above.
Then the induced matrix 
$\psi_4: \sum_{i=1}^5\cC_4^{b_i}(-3)\to\sum_{i=1}^5\cC_4^{a_i}(-2)$ has a six dimensional kernel.
\end{proposition}
\begin{proof}
According to section \ref{section2}, we can write down the relevant cohomology groups. Let 
$\{s,t\}$ be a basis of $G=H^0(\PPE,\OO_\PPE(R))$ and $\{\varphi_1\}$ be a basis of 
$H^0(\PPE,\OO_\PPE(H-3R))$ then a basis of $H^0(\PPE,\OO_\PPE(H-2R))$ is given by 
$\{s\varphi_1,t\varphi_1,\varphi_2,\varphi_3,\varphi_4\}$. We consider the submatrix 
$\psi_{(41)}:\bigwedge^4F\otimes S_1G(-7)\to\left(\bigwedge^5F(-7)\right)^4$ of $\psi_4$
induced by the first column of $\Psi$. As in the proof of Theorem \ref{mthm}, the map $\psi_{(41)}$ is given as the composition
\begin{center}$
\begin{xy}
\xymatrix{
\bigwedge^4F\otimes S_1G \cong \bigwedge^4F\otimes H^0(\OO_\PPE(R))
\ar[d]^{id\otimes \Psi_{(41)}} \\
 \bigwedge^4F\otimes H^0(\OO_\PPE(H-R))^{\oplus 4}\cong\bigwedge^4F\otimes F^{\oplus 4}
 \ar[r]^{\ \ \ \ \ \ \ \ \ \ \ \ \ \ \ \ \ \ \  \wedge} 
 &(\bigwedge^5F)^{\oplus 4}  .
}
\end{xy}
$\end{center}
 By our generality assumption on $C$, we can assume that the $4$ entries of $\Psi_{(41)}$ are independent and after acting with an element in $\Aut(X)$, we can furthermore assume that 
$\Psi_{(41)}=(s\varphi_1,\varphi_2,\varphi_3,\varphi_4)^t$. It follows that elements of the form 
$$
(\lambda s +\mu t)s\varphi_1\wedge (\lambda s +\mu t)\varphi_2\wedge (\lambda s +\mu t)\varphi_3\wedge (\lambda s +\mu t)\varphi_4\otimes (\lambda s +\mu t),\text{with}\  \lambda,\mu\in\Bbbk
$$
lie in the kernel of $\psi_{(41)}$. Expanding those elements we get
 $$
 \lambda^5s^2\varphi_1\wedge s\varphi_2\wedge s\varphi_3\wedge s\varphi_4\otimes s+\cdots + \mu^5st\varphi_1\wedge t\varphi_2\wedge t\varphi_3\wedge t\varphi_4\otimes t
 $$
 and conclude that a rational normal curve of degree $5$ lies in $\PP(\text{Syz})$ where 
 $\text{Syz}\subset\Tor_6^T(T/I_C)_{7}$ is the subspace of the 
 $6^{\text{th}}$ syzygy module spanned by the extra syzygies and $I_C\subset T$ denotes the ideal of the canonical curve $C$. We get $\beta_{6,7}(C)\geq \beta_{6,7}(X)+6=222$ and by computing one example using \emph{Macaulay2}, it follows that $\psi_{4}$ has a $6$-dimensional kernel in general. 
\end{proof}

\begin{remark}
A direct computation using \emph{Macaulay2} shows that none of the entries of the skew symmetric matrix $\Psi$ can be made zero by suitable row and column operations respecting the skew symmetric structure of $\Psi$. By \cite[Section 5]{Sch} this implies  that the $6$ extra syzygies are not induced by an additional linear series on $C$.
\end{remark}

It arises the question how the extra syzygies of a $5$-gonal canonical curve $C\subset \PP^{12}$ differ from the syzygies induced by the scroll swept out by the $g^1_5$ on $C$. At least in the genus $g=13$ case we can give an answer in this direction by introducing the notion of \emph{syzygy schemes}, originally defined in \cite{Ehbauer}. 

\begin{defrmk}
Let $C\subset \PP^{g-1}$ be an smooth and irreducible canonical curve and let $I_C\subset S$ be the ideal of $C$. Let further
\begin{center}$
\textbf{\upshape{L}}_\bullet: S \leftarrow S(-2)^{\beta_{1,2}}\leftarrow S(-3)^{\beta_{2,3}}\leftarrow \cdots
$\end{center}
be the linear strand of a minimal free resolution of $S_C=S/I_C$.
For a $p^{\text{th}}$ linear syzygy 
$s\in L_p$, let $V_s$ be the smallest vector space such that the following diagram commutes
\begin{align*}
L_{p-1} \ \ \ \ \ &\longleftarrow \  \ \ \ \  L_p \\
\cup \ \ \ \ \ \ \ &\ \ \ \ \ \ \ \ \ \ \ \  \cup \\[-2mm]
V_s\otimes S(-p) \ &\longleftarrow \ S(-p-1)\cong \langle s\rangle \ .
\end{align*}
The \emph{rank of the syzygy} $s$ is defined to be $\rank(s):=\dim V_s$. \newline
Since $\Hom(\textbf{\upshape{L}}_\bullet,S)$ is a free complex and the Koszul complex is exact, it follows that the maps of the dual diagram extend to a morphism of complexes. Dualizing again we get 
\begin{center}$
\begin{xy}
\xymatrix{
&S&L_1 \ar[l] &\ar[l] \cdots \ar[l]  & \ar[l] L_{p-1}&\ar[l] L_p\\
&\bigwedge^{p}V_s\otimes S(-1) \ar[u]_{\varphi_{p}}& \ar[l] \bigwedge^{p-1}V_s\otimes S(-2)\ar[u] & \ar[l] \cdots  
& \ar[l]V_s\otimes S(-p) \ar[u] & \ar[l] S(-p-1)\ar[u] \ .
}
\end{xy}
$\end{center}
By degree reasons there are only trivial homotopies and therefore all the vertical maps except $\varphi_{p}$ are unique. The \emph{syzygy scheme} $\Syz(s)$ of $s$ is the the scheme defined by the ideal 
$$
I_s= \im(S\longleftarrow\bigwedge^{p-1} V_s\otimes S(-2)).
$$
The $p^{\text{th}}$-\emph{syzygy scheme} $\Syz_p(C)$ of a curve $C$  is defined to be the intersection
$
\bigcap_{s\in L_p} \Syz(s).
$
\end{defrmk}
Any $p^{\text{th}}$-syzygy of a canonical curve has rank $\geq p+1$ and the
syzygies of rank $p+1$ are called scrollar syzygies. The name is justified by a theorem due to von Bothmer:

\begin{theorem}[\cite{bothmer2}]
Let $s\in L_p$ be a $p^{\text{th}}$ scrollar syzygy. Then $\Syz(s)$ is a rational normal scroll of degree $p+1$  and codimension $p$ that contains the curve $C$.
\end{theorem}

We can now come back to our example of a $5$-gonal genus $13$ curve. Recall that in this case the space of extra syzygies can be identified with the kernel of the map
\begin{center}$
\cC^{5}_4(-3) \to (\cC^{3}_4)^{\oplus 4}(-2)
$\end{center}
 which is induced by the first column of the skew symmetric matrix 
$\Psi$.\newline
We denote by $\Pf_1,\dots,\Pf_4 \in H^0(\OO_\PPE,\OO_\PPE(2H-3R))$ the $4$ Pfaffians of the matrix $\Psi$ that involve the first column and consider the iterated mapping cone 
$$
\left[\left[\cC^5 \to (\cC^3)^{\oplus 4}\right]\to\cC^0\right],
$$
where $\sum_{i=1}^4\cC^3\to \cC^0$ is induced by the multiplication with $(\Pf_1,\dots,\Pf_4)$.
This complex is a resolution of the ideal $J$ generated by the $4$ Pfaffians as an $\OO_{\PP^{12}}$-module. In particular the minimized resolution is a subcomplex of the minimal free resolution of $S_C$. \newline
Since all extra syzygies are induced by the first column of $\Psi$, it follows that the $6^{\text{th}}$ syzygy modules in the linear strand of these minimal resolutions are canonically isomorphic.
 Therefore $\Syz_6(V(J))$ and $\Syz_6(C)$ coincide and $V(J)\subset\Syz_6(C)$.

 \begin{proposition}\label{propSyzIC}
 Let $C\subset \PP^{12}$ be a general $5$-gonal canonical curve. Then 
 $\Syz_6(C)$ is the scheme cut out by the $4$ Pfaffians of $\Psi$ involving the first column. In particular 
 \begin{center}$\Syz_6(C)=C\cup p$\end{center}
 for some point $p\in X$.
 \end{proposition}
 \begin{proof}
 One inclusion follows from the discussion above. 
 The other inclusion follows by computing one example using \emph{Macaulay2}.
 \end{proof}
 
\begin{remark}
Since we do not have a full description of the space of extra syzygies for $5$-gonal canonical curves of higher genus, 
a similar approach does not work in these cases.
\end{remark}

\medskip

\begin{acknowledgement}
This article is based on my Master's thesis. First of all I would like to thank my advisor Frank-Olaf Schreyer for valuable discussions on the topic and continuous support.
I would further like to thank Michael Hahn for carefully reading this article and pointing out many typos and misprints.
\end{acknowledgement}

\bibliographystyle{alpha}
\bibliography{mybib}
\end{document}